\documentclass[12pt]{amsart}
\usepackage{amssymb}
\usepackage{amsmath}
\usepackage{amsthm}
\newtheorem{theorem}{Theorem}[section]
\newtheorem{lemma}[theorem]{Lemma}

\newtheorem{definition}[theorem]{Definition}

\newtheorem{corollary}[theorem]{Corollary}
\newtheorem{proposition}[theorem]{Proposition}

\newtheorem{lem-def}[theorem]{Lemma-Definition}
\DeclareRobustCommand\longtwoheadrightarrow
{\relbar\joinrel\twoheadrightarrow}

\newcommand{\hooklongrightarrow}{\lhook\joinrel\longrightarrow}

\newcommand{\Q}{\mathbb Q}

\def\op{\operatorname}

\def\al{\alpha}
\def\as#1{\renewcommand\arraystretch{#1}}

\def\be{\beta}

\def\chr{\op{char}}

\def\diso{\lower.4ex\hbox{$\downarrow$}\raise.4ex\hbox{\mbox{\scriptsize
$\wr$}}}

\def\e{\medskip}

\def\g{\Gamma}
\def\ga{\gamma}
\def\gal{\op{Gal}}

\def\gg{\mathcal{G}}
\def\ggm{\mathcal{G}_\mu}

\def\gm{\g_\mu}

\def\gr{\operatorname{gr}}

\def\hm{H_\mu}

\def\imp{\,\Longrightarrow\,}

\def\iso{\ \lower.3ex\hbox{\as{.08}$\begin{array}{c}\lra\\\mbox{\tiny $\sim\,$}\end{array}$}\ }

\def\kb{\overline{K}}

\def\kx{K[x]}

\def\la{\lambda}

\def\lg{l\raise.6ex\hbox to.2em{\hss.\hss}l}

\def\lra{\,\longrightarrow\,}

\def\minf{\mu_{-\infty}}

\def\mn{\op{Min}}

\def\mx{\op{Max}}

\def\om{\omega}

\def\orb{\hbox to  .3em{$\backslash$}\backslash}

\def\ppa{\mathcal{P}_{\alpha}}
\def\pset{\mathcal{P}}

\def\qg{\mathbb{Q}\g}

\def\ram{^{\mbox{\tiny ram}}}

\def\sep{_{\mbox{\tiny sep}}}

\def\t{\theta}

\def\Min{\mathrm{Min}}

\newcounter{cs}
\stepcounter{cs}
\newcommand{\casos}{\begin{itemize}}
\newcommand{\fcasos}{\end{itemize}\setcounter{cs}{1}}

\newfont{\tit}{cmr12 scaled \magstep3}

\title[Invariants of algebraic elements]{Invariants of algebraic elements over henselian fields}

\makeatletter
\@namedef{subjclassname@2010}{%
  \textup{2010} Mathematics Subject Classification}

\author[Moraes de Oliveira]{Nath\'alia Moraes de Oliveira}
\address{Departament de Matem\`{a}tiques,
         Universitat Aut\`{o}noma de Barcelona,
         Edifici C, E-08193 Bellaterra, Barcelona, Catalonia, Spain}
\email{noliveira@mat.uab.cat}

\thanks{Partially supported by grant 204224/2014-4 from CNPq}
\date{}

\begin{document}

\begin{abstract}
Let $(K,v)$ be a henselian valued field. In this paper, we use Okutsu sequences for monic, irreducible polynomials in $\kx$, and their relationship with MacLane chains of inductive valuations on $\kx$, to obtain some results  on the computation of invariants of algebraic elements over $K$.
\end{abstract}

\maketitle

\section*{Introduction}
Let $(K,v)$ be a henselian field and denote still by $v$ the canonical extension of this valuation to a fixed algebraic closure $\kb$ of $K$. 
Let $\g=v(K^*)$ be the value group of $v$ over $K$, and denote by $\qg=\g\otimes\Q$ the value group of $v$ over $\kb$.

Let $F\in\kx$ be a monic, irreducible polynomial  of degree $n>1$. Denote by $Z(F)\subset\kb$ the set of rootos of $F$, and take $\t\in Z(F)$. Consider the following set
\begin{equation}\label{wset}
W(\t)=\left\{\dfrac{v(g(\t))}{\deg(g)}\;\Big|\; g\in\kx\mbox{ monic},\ 0<\deg(g)<n\right\}\subset\qg.
\end{equation}

If this set contains a maximal element, we define the \emph{weight} of $\t$ (or $F$) as:
$$
w(F)=w(\t)=\mx\left(W(\t)\right).
$$

A pair  $\phi,F$ is a \emph{distinguished pair} of polynomials if $\phi\in\kx$ is a monic polynomial of minimal degree satisfying $v(\phi(\t))/\deg(\phi)=w(F)$.

An \emph{Okutsu frame} of $F$ is a list $[\phi_0,\phi_1,\dots,\phi_r]$ of monic, irreducible polynomials in $\kx$, such that $\deg(\phi_0)=1$ and $\phi_i,\phi_{i+1}$ is a distinguished pair for all $0\le i\le r$, where we agree that $\phi_{r+1}=F$.

If $F$ admits an Okutsu frame, we say that $F$ is an \emph{Okutsu polynomial}, and the index $r\ge0$ is called the \emph{Okutsu depth} of $F$. 

A monic, irreducible polynomial $F\in\kx$ is an Okutsu polynomial if and only if it is defectless; that is, $\deg(F)=e(K(\t)/K)f(K(\t)/K)$. As a consequence, there is a tight link between defectless polynomials and key polynomials for inductive valuations on $\kx$ \cite[Sec. 10.4]{IndVals}. In particular, we may attach to any defectless polynomial $F$ a bunch of arithmetic invariants, which can be read in any optimal MacLane chain of a certain inductive valuation on $\kx$ canonically associated with $F$ (cf. section \ref{secIndVals}).  

In Corollary \ref{weight} we give an explicit formula for the weight of a defectless polynomial in terms of these invariants. \e 

On the other hand, the value $v(g(\t))/\deg(g)$ is the average of the values $v(\t-\beta)$ for $\beta$ running on the roots of $g$ in $\kb$. Let us focus our attention on these latter values, and consider the set
\begin{equation}\label{Dset}
\Delta(\t)=\{v(\t-\al)\mid \al\in\kb,\ \deg_K(\al)<\deg(\t)\}.
\end{equation}

If $F$ is defectless, this set contains a maximal value which is called the 
\emph{main invariant of} $\t$ (or $F$), and is denoted
$$
\delta(F)=\delta(\t)=\mx\left(\Delta(\t)\right).
$$

In this paper, we use the connection between defectless polynomials and MacLane chains of inductive valuations to prove that, if $\t$ is quasi-tame over $K$, then the main invariant coincides with Krasner's constant:
$$
\om(F)=\om(\t)=\mx\left(\Omega(\t)\right),\qquad \Omega(\t)=\{v(\t-\t')\mid \t'\in\op{Z}(F),\ \t'\ne\t\},
$$
where we consider $\Omega(\t)$ as a multiset of cardinality $n-1$. 

Also, we find explicit formulas for all values in the multiset $\Omega(\t)$, and for their multiplicities, in terms of the discrete invariants supported by any optimal MacLane chain corresponding to $F$.

Most of these results can be found in the literature as the combined contribution of several authors. Special mention deserve Aghigh-Khanduja \cite{MainInvSudesh,CHF}, Brown-Merzel \cite{BrownMerzel,BrownMerzel2} and Singh-Khanduja \cite{TameExtKandhuja}.

Our approach yields more direct proofs and a new arithmetic interpretation of the values of the multiset  $\Omega(\t)$ in the case that $\t$ is quasi-tame.

\section{Inductive valuations and Okutsu frames - A short background}\label{IndOku}

We keep dealing with our henselian valued field $(K,v)$.

In this section, we review some basic facts on valuations on $K[x]$, mainly extracted from \cite{Vaq} and \cite{IndVals}.

\subsection{Key polynomials and augmentation of valuations}
Let $\mu$ be a valuation on $K(x)$ extending $v$.

Let $\g_\mu=\mu\left(K(x)^*\right)$ be the value group, and $k_{\mu}$
 the residue class field of $\mu$.

For any $\alpha\in\g_\mu$, consider the abelian groups:
$$
\ppa=\{g\in \kx\mid \mu(g)\ge \alpha\}\supset
\ppa^+=\{g\in \kx\mid \mu(g)> \alpha\}.
$$  

The \emph{graded algebra of $\mu$ over $\kx$} is the integral domain:
$$
\ggm:=\gr_{\mu}(\kx)=\bigoplus\nolimits_{\alpha\in\g_\mu}\ppa/\ppa^+.
$$

Let $\Delta_\mu=\pset_0/\pset_0^+\subset\ggm$ be the subring of homogeneous elements of degree zero. 
There are canonical injective ring homomorphisms: 
$$ k\hooklongrightarrow\Delta_\mu\hooklongrightarrow k_{\mu}.$$
In particular, $\Delta_\mu$ and $\ggm$ are equipped with a canonical structure of $ k$-algebra. 

There is a natural map $\hm\colon \kx\to \ggm$, given by $\hm(0)=0$ and
$$\hm(g)= g+\pset_{\mu(g)}^+\in\pset_{\mu(g)}/\pset_{\mu(g)}^+, \mbox{ if }g\ne0.
$$

\begin{definition}\label{mu}
A \emph{key polynomial} for $\mu$ is a monic polynomial $\phi\in \kx$ such that 
\begin{itemize}
\item $\hm(\phi)$ is a prime element in $\ggm$.
\item For all $f\in\kx$ with $\deg(f)<\deg(\phi)$, the prime element $\hm(\phi)$ does not divide $\hm(f)$ in $\ggm$. 
\end{itemize}
A key polynomial is necessarily irreducible in $\kx$.
\end{definition}

\begin{theorem}\cite[Thm. 3.9]{KeyPol}\label{bound}
Let $\phi\in\kx$ be a key polynomial for $\mu$. For any monic non-constant $f\in \kx$ we have
$$
\mu(f)/\deg(f)\le w(\mu):=\mu(\phi)/\deg(\phi),
$$
and equality holds if $f$ is a key polynomial for $\mu$.

We say that $w(\mu)$ is the \emph{weight} of $\mu$.
\end{theorem}

Consider an order-preserving embedding $\iota\colon \gm\hookrightarrow \g'$ of  ordered  abelian groups. Take $\phi\in \kx$ a key polynomial for $\mu$, and $\ga\in \g'$ any element such that $\mu(\phi)<\ga$. 

For any $f\in\kx$, consider its canonical $\phi$-expansion
$$
f=\sum_{0\le s}a_s\phi^s,\qquad a_s\in\kx,\quad\deg(a_s)<\deg(\phi).
$$
Then, the following mapping is a valuation on $\kx$:
$$
\mu'\colon\kx\rightarrow \g' \cup \left\{\infty\right\},\qquad \mu'(f)=\mn\left\{\mu(a_s)+s\ga\mid 0\le s\right\}.
$$
We say that $\mu'=[\mu;\phi,\ga]$ is an \emph{augmented valuation} of $\mu$. It satisfies:
$$
\mu(\phi)<\ga=\mu'(\phi),\qquad \mu(f)\le \mu'(f),\ \forall\,f\in\kx.
$$
Hence, we have a natural homomorphism of graded algebras $\ggm\to\gg_{\mu'}$.

\begin{lemma}\label{imagedelta}
Let $\mu'=[\mu;\phi,\ga]$. Then,  $\phi$ is a key polynomial for $\mu'$. 

Moreover, let $K_\phi=K[x]/{\phi K[x]}$ and consider the semivaluation:
$$
v_{\phi}\colon \kx \longtwoheadrightarrow K_\phi \stackrel{v}\lra \qg\cup\{\infty\}.$$
Then, the group of values  $\g_{v_\phi}$ is equal to $\gm$.
\end{lemma}

\subsection{Inductive valuations}\label{secIndVals}

A valuation $\mu$ on $\kx$, extending $v$, is said to be \emph{inductive} if it is attained after a finite number of augmentation steps:
\begin{equation}\label{depth}
\minf\stackrel{\phi_0,\ga_0}\lra\  \mu_0\ \stackrel{\phi_1,\ga_1}\lra\  \mu_1\ \stackrel{\phi_2,\ga_2}\lra\ \cdots
\ \stackrel{\phi_{r-1},\ga_{r-1}}\lra\ \mu_{r-1} 
\ \stackrel{\phi_{r},\ga_{r}}\lra\ \mu_{r}=\mu,
\end{equation}
with $\ga_0,\dots,\ga_r\in\qg$, and intermediate valuations $\mu_i=[\mu_{i-1};\phi_i,\ga_i]$, for $0< i\leq r$.

The valuation $\minf$ is an incommensurable extension of $v$ to $\kx$, playing the role of absolute minimal extension. Since we are not going to use it, let us simply say that its key polynomial $\phi_0\in\kx$ has degree one, and its augmentation $\mu_0$ is defined as:
$$
\mu_0\left(\sum\nolimits_{0\le s}a_s\phi_0^s\right)=\mn\left\{v(a_s)+s\ga_0\mid 0\le s\right\}.
$$

Denote $m_i=\deg(\phi_i)$ for $0\le i\le r$.

An \emph{optimal MacLane chain} of $\mu$ is any chain (\ref{depth}) of augmentations satisfying  
$$
1=m_0\mid\cdots\mid m_r,\qquad m_0<\cdots<m_r.
$$

All inductive valuations admit optimal MacLane chains. These chains are not unique, but they support many intrinsic data of $\mu$:
\begin{itemize}
\item The intermediate valuations $\mu_0,\dots,\mu_{r-1}$.
\item The degrees $m_0,\dots,m_r$ of the key polynomials.
\item The \emph{slopes} $\ga_0,\cdots,\ga_r$, which satisfy $\ga_i=\mu_i(\phi_i)=\mu(\phi_i)$ for all $0\le i \le r$.
\item The \emph{secondary slopes} $\la_0,\cdots,\la_r$, defined as $$\la_0=\ga_0=\mu(\phi_0),\qquad \la_i=\mu_i(\phi_i)-\mu_{i-1}(\phi_i)>0, \quad 0<i\le r.$$
\item The relative ramification indices $e_0,\dots, e_{r-1}$, defined as
$$
e_0=1,\qquad e_i=\left(\g_{\mu_i}\colon \g_{\mu_{i-1}}\right),\quad 0<i\le r.
$$
\end{itemize}


For any key polynomial $\phi$ for $\mu$, there is a tower of fields:
\begin{equation}\label{towerki}
k\simeq  k_{\phi_0}\lra  k_{\phi_1}\lra\cdots\lra  k_{\phi_r}\lra k_\phi.
\end{equation}

 The identification $\g_{\mu_{i-1}}=\g_{v_{\phi_i}}$, given in Lemma \ref{imagedelta}, allows a computation of the ramification index of the extension $K_{\phi_i}/K$ in terms of these data: 
\begin{equation}\label{ephi}
e(\phi_i):=e\left(K_{\phi_i}/K\right)=\left(\g_{\mu_{i-1}}\colon \g\right)=e_0\cdots e_{i-1}.
\end{equation}

Finally, Lemma \ref{imagedelta} shows that each intermediate valuation $\mu_i$ admits $\phi_i$ and $\phi_{i+1}$ as key polynomials. By Theorem \ref{bound},
$$
\dfrac{\ga_i}{m_i}=\dfrac{\mu_i(\phi_i)}{m_i}=w(\mu_i)=\dfrac{\mu_i(\phi_{i+1})}{m_{i+1}}=\dfrac{\ga_{i+1}-\la_{i+1}}{m_{i+1}}.
$$
This relates the main and secondary slopes by an explicit formula:
\begin{equation}\label{recurrence}
\dfrac{\ga_i}{m_i}=\dfrac{\la_0}{m_0}+\cdots+\dfrac{\la_i}{m_i},\qquad 0\le i\le r.
\end{equation}

\subsection{Okutsu frames}
Let us quote some fundamental results of \cite{IndVals}.

\begin{theorem}\label{MLOk}
Consider an optimal MacLane chain of an inductive valuation $\mu$ as in (\ref{depth}).
Let $F$ be a key polynomial for $\mu$. Then, $F$ is an Okutsu polynomial, and
\begin{enumerate}
\item[(1)] If $\deg(F)>\deg(\phi_r)$, then $[\phi_0,\dots,\phi_r]$ is an Okutsu frame of $F$. 
\item[(2)] If $\deg(F)=\deg(\phi_r)$, then $[\phi_0,\dots,\phi_{r-1}]$ is an Okutsu frame of $F$.
\end{enumerate}
Moreover, $w(F)=w(\mu)$ if $\deg(F)>\deg(\phi_r)$.  
\end{theorem}

\begin{theorem}\label{OkML}
Let $F\in\kx$ be an Okutsu polynomial, and let $\t\in\kb$ be a root of $F$. 
Let $[\phi_0,\dots,\phi_r]$ be an Okutsu frame of  $\phi_{r+1}=F$. 
For all $0\le i\le r$, denote $\ga_i=v(\phi_i(\t))$ and consider the mapping
$$
\mu_i\colon \kx\lra\qg\cup\{\infty\}, \qquad \sum\nolimits_{0\le s}a_s\phi_i^s\ \longmapsto \ \mn\{v(a_s(\t))+s\ga_i\mid 0\le s\},
$$
where $\deg(a_s)<\deg(\phi_i)$ for all $s\ge0$.

Then, $\mu_i$ is a valuation, $\phi_{i+1}$ is a key polynomial for $\mu_i$, and $\mu_r$   admits an optimal MacLane chain  
$$
\minf\stackrel{\phi_0,\ga_0}\lra\  \mu_0\ \stackrel{\phi_1,\ga_1}\lra\  \mu_1\ \stackrel{\phi_2,\ga_2}\lra\ \cdots
\ \stackrel{\phi_{r-1},\ga_{r-1}}\lra\ \mu_{r-1} 
\ \stackrel{\phi_{r},\ga_{r}}\lra\ \mu_{r}.
$$
\end{theorem}

The next result follows from Theorems \ref{MLOk}, \ref{OkML} and from \cite{Vaq2}.

\begin{theorem}\label{Ok=D1}
Let $F$ be a monic irreducible polynomial in $K[x]$. The following conditions are equivalent:
\begin{enumerate}
\item[(1)] $F$ is the key polynomial of an inductive valuation.
\item[(2)] $F$ is an Okutsu polynomial.
\item[(3)] $F$ is defectless.
\end{enumerate}
\end{theorem}

\begin{corollary}\label{weight}
Let $F\in\kx$ be a monic, irreducible defectless polynomial. Then, if $r$ is the Okutsu depth of $F$, we have
$$
w(F)=\dfrac{\ga_r}{m_r}=\dfrac{\la_0}{m_0}+\cdots+\dfrac{\la_r}{m_r}.
$$
where $\ga_i$, $m_i$, $\la_i$ are the intrinsic data of any optimal MacLane chain of the inductive valuation attached to $F$ in Theorem \ref{Ok=D1}.
\end{corollary}

\section{Complete distinguished chains of defectless algebraic elements}\label{secDChain=Okutsu}
\pagestyle{headings}

We keep dealing with a henselian valued field $(K,v)$.

In section \ref{secMinPairs} we connect distinguished pairs of polynomials with  distinguished pairs of algebraic elements. 

Distinguished pairs of algebraic elements and complete distinguished chains were introduced by N. Popescu-A. Zaharescu in 1995, for $K$ a complete, discrete, rank-one valued field  \cite{PZ}. However, these objects are equivalent to some sequences of algebraic elements studied by Okutsu in 1982, also in the complete and discrete rank-one case  \cite{Ok}. 
In section \ref{subsecDChain=Okutsu} we show the equivalence between the two concepts, for arbitrary henselian fields.

\subsection{Distinguished pairs of algebraic elements}\label{secMinPairs}
Let $F\in\kx$ be a monic, irreducible polynomial of degree $n>1$, and let $\t\in\kb$ be a rot of $F$.

In this section, we prove that the set $W(\t)$ in (\ref{wset}) contains a maximal value if and only if the set $\Delta(\t)$ in (\ref{Dset}) contains a maximal value. 

\begin{lemma}\label{InsepSep}
For any $\beta\in\kb$ inseparable over $K$, and any $\rho\in\qg$, there exists $\beta\sep\in \kb$ separable over $K$ such that
$$\deg_K(\beta\sep)=\deg_K(\beta),\qquad v(\beta-\beta\sep)>\rho.$$ 
\end{lemma}

\begin{proof}

Let $g\in\kx$ be the minimal polynomial of $\beta$ over $K$. We have $g'=0$.

Consider the polynomial $g\sep=g+\pi x\in\kx$, where $\pi\in K^*$ satisfies $$v(\pi)>\deg_K(\beta)\,\rho-v(\beta).$$ Since $g\sep'=\pi\ne0$, this polynomial is separable.  On the other hand,
$$
\sum\nolimits_{\al\in\op{Z}(g\sep)}v(\beta-\al)=v\left(g\sep(\beta)\right)=v(\pi\beta)=v(\pi)+v(\beta)> \deg_K(\beta)\,\rho.
$$
Hence, there exists $\al\in\op{Z}(g\sep)$ such that $v(\beta-\al)>\rho$. We may take $\beta\sep=\al$.
\end{proof}

\begin{definition}
Let $\al\in\kb$ with $\deg_K(\al)<n$. 

We say that $\al,\,\t$ is a distinguished pair if 
the two following conditions are satisfied:
\begin{enumerate}
\item[(1)] \ $v(\t-\al)=\delta(\t)$. 
\item[(2)] \ $\beta\in\kb,\quad \deg_K(\beta)<\deg_K(\al)\ \imp\  v(\t-\beta)<\delta(\t)$.
\end{enumerate}
\end{definition}

\begin{theorem}\label{min=min}\mbox{\null}

\begin{enumerate}
\item[(1)] Suppose that $\phi,\,F$ is a distinguished pair of polynomials. If $\al\in\op{Z}(\phi)$ has $v(\t-\al)=\mx\{v(\t-\al')\mid \al'\in\op{Z}(\phi)\}$, then
 $\al,\t$ is a distinguished pair. 
\item[(2)] Suppose that $\al,\,\t$ is a distinguished pair. If $\phi\in\kx$ is the minimal polynomial of $\al$ over $K$, then $\phi,\,F$ is a distinguished pair of polynomials. 
\end{enumerate}
\end{theorem}

\begin{proof}
Let us first check (1). Suppose that $\phi,\,F$ is a distinguished pair of polynomials. Let $\delta=v(\t-\al)=\mx\{v(\t-\al')\mid \al'\in\op{Z}(\phi)\}$.

Consider any $\beta\in\kb$ with $\deg_K(\beta)<n$. We want to show:
\begin{enumerate}
 \item[(i)\,]  \ $v(\t-\beta)\le\delta$.
 \item[(ii)] \ $v(\t-\beta)=\delta\ \imp\ \deg_K(\beta)\ge\deg_K(\al)$.
\end{enumerate}

Let $g\in\kx$ be the minimal polynomial of $\beta$ over $K$. We may assume that $v(\t-\beta)=\mx\{v(\t-\beta')\mid\beta'\in\op{Z}(g)\}$. By Lemma \ref{InsepSep}, we may assume too, that $\t$, $\al$ and $\beta$ are separable over $K$. Consider a finite Galois extension  $M/K$ containing $\t$, $\al$ and $\beta$, and denote $G=\gal(M/K)$. We claim that
\begin{equation}\label{claimtg}
v(\t-\beta)\ge\delta\quad\imp\quad \dfrac{v(g(\t))}{\deg(g)}\,\ge\,\dfrac{v(\phi(\t))}{\deg(\phi)}.
\end{equation}

In fact, assume that $v(\t-\beta)\ge\delta$. Then, for any $\sigma\in G$ we get:
\begin{equation}\label{inici}
\as{1.3}
\begin{array}{rl}
v(\t-\sigma(\beta))&=\ \, v(\t-\sigma(\al)+\sigma(\al)-\sigma(\t)+\sigma(\t)-\sigma(\beta))\\
&\ge\ \,\mn\{v(\t-\sigma(\al)),\,v(\sigma(\al)-\sigma(\t)),\,v(\sigma(\t)-\sigma(\beta))\}\\
&=\ \,\mn\{v(\t-\sigma(\al)),\,v(\al-\t),\,v(\t-\beta)\}=v(\t-\sigma(\al)),
\end{array}
\end{equation}
because $v(\t-\sigma(\al))\le\delta$, while $v(\al-\t),\,v(\t-\beta)\ge\delta$. Therefore,
\begin{equation}\label{compare}
\dfrac{\#G}{\deg(g)}\,v(g(\t))=\sum_{\sigma\in G}v(\t-\sigma(\beta))\ge\sum_{\sigma\in G}v(\t-\sigma(\al))= \dfrac{\#G}{\deg(\phi)}\,v(\phi(\t)).
\end{equation}
This proves the claimed implication (\ref{claimtg}).

Now, if we had $v(\t-\beta)>\delta$, then at least for the automorphism $\sigma=1$ we would have $v(\t-\sigma(\beta))>\delta=v(\t-\sigma(\al))$, leading to a strict inequality in (\ref{compare}). This would contradict the fact that $\phi,F$ is a distinguished pair. This argument proves (i).\e

On the other hand, the equality $v(\t-\beta)=\delta$ is incompatible with a strict inequality in (\ref{compare}). In fact, suppose that for some $\sigma\in G$ we had
$$
\delta=v(\t-\beta)\ge v(\t-\sigma(\beta))>v(\t-\sigma(\al)).
$$
Then, the inequality in (\ref{inici}) becomes an equality, and this contradicts our assumptions:
$$
v(\t-\sigma(\beta))=v(\t-\sigma(\al)).
$$
Thus, if $v(\t-\beta)=\delta$, we must have an equality in (\ref{compare}). Since $\phi,F$ is a distinguished pair, this implies $\deg(g)\ge\deg(\phi)$. This proves (ii).\e

Let us now prove (2). Suppose that $\al,\,\t$ is a distinguished pair, and keep the notation \ $\delta=v(\t-\al)=\mx\{v(\t-\al')\mid \al'\in\op{Z}(\phi)\}$.

Let $g\in\kx$ be a monic polynomial with $\deg(g)<n$. We want to show:
\begin{enumerate}
 \item[(i)\,]  \ $v(g(\t))/\deg(g)\,\le\,v(\phi(\t))/\deg(\phi)$.
 \item[(ii)] \ $v(g(\t))/\deg(g)=v(\phi(\t))/\deg(\phi)\ \imp\ \deg(g)\ge\deg(\phi)$.
\end{enumerate}

By Lemma \ref{Irr+Sep} below, we may assume that $g$ is irreducible and separable. Also, by Lemma \ref{InsepSep}, we may assume that $\al$ and $\t$ are separable too. Let $M/K$ be a finite Galois extension containing $\t$, $\al$ and $\beta$, and denote $G=\gal(M/K)$.

Take $\beta\in\op{Z}(g)$ such that $v(\t-\be)=\mx\{v(\t-\be')\mid \be'\in\op{Z}(g)\}$, then for all $\sigma\in G$, 
\begin{equation}\label{inici2}
\as{1.3}
\begin{array}{rl}
v(\t-\sigma(\t))&=\ v\left(\t-\sigma(\be)+\sigma(\be)-\sigma(\t)\right)\\&\ge\; \mn\{v(\t-\sigma(\be)),\,v(\sigma(\be)-\sigma(\t))\}\\&=\;\mn\{v(\t-\sigma(\be)),\,v(\be-\t)\}=v(\t-\sigma(\be)).
\end{array}
\end{equation}

Now, we claim that $$v\left(\t-\sigma(\be)\right)\le v\left(\t-\sigma(\al)\right), \qquad\forall\,\sigma\in G.$$ In fact, if $v\left(\t-\sigma(\al)\right)=\delta$, then our assumption is a consequence of the fact that $\al,\,\t$ is a distinguished pair. If $v\left(\t-\sigma(\al)\right)<\delta$, then the claim follows from (\ref{inici2}). Indeed,
$$
v\left(\t-\sigma(\be)\right)\le v\left(\t-\sigma(\t)\right)=v\left(\t-\sigma(\al)+\sigma(\al)-\sigma(\t)\right)=v\left(\t-\sigma(\al)\right),
$$
because $v\left(\t-\sigma(\al)\right)<\delta=v\left(\sigma(\al)-\sigma(\t)\right)$.

From the claim it follows that
\begin{equation}\label{compare2}
\dfrac{\#G}{\deg(g)}\,v(g(\t))=\sum_{\sigma\in G}v(\t-\sigma(\beta))\le\sum_{\sigma\in G}v(\t-\sigma(\al))= \dfrac{\#G}{\deg(\phi)}\,v(\phi(\t)).
\end{equation}

Also, if equality holds in (\ref{compare2}), then $v(\t-\sigma(\beta))=v(\t-\sigma(\al))$, for all $\sigma\in G$. 
In particular, for $\sigma=1$ we deduce $v(\t-\be)=v(\t-\al)$, which implies 
$$
\deg(g)=\deg_K(\be)\ge\deg_K(\al)=\deg(\phi),
$$
because $\al,\,\t$ is a distinguished pair. This proves (ii).
\end{proof}

\begin{lemma}\label{Irr+Sep}
Let $\phi,\,F\in\kx$ be monic, irreducible polynomials with $\deg(\phi)<\deg(F)$. Then, for $\phi,\,F$ to be a distinguished pair it suffices to check that the two conditions:
\begin{enumerate}
\item[(i)\,] \ $0<\deg(g)<\deg(F)\ \imp\ v(g(\t))/\deg(g)\,\le\,v(\phi(\t))/\deg(\phi)$,
\item[(ii)] \ $v(g(\t))/\deg(g)\,=\,v(\phi(\t))/\deg(\phi)\ \imp\ \deg(g)\ge\deg(F)$,
\end{enumerate}
hold for all monic, irreducible and separable polynomials $g\in\kx$. 
\end{lemma}

\begin{proof}
Let us first show that if conditions (i), (ii)  hold for all monic irreducible polynomials in $\kx$, then both conditions hold for all monic polynomials.

Let $g=h_1\cdots h_t$ be a product of monic (not necessarily different) irreducible polynomials. Clearly, the average of the values $v(\t-\be)$ for $\be\in\op{Z}(g)$ is less than, or equal to, the maximum of the averages of the values $v(\t-\be)$, taken on the subsets 
$\op{Z}(g)=\op{Z}(h_1)\cup\cdots\cup \op{Z}(h_t).$
In other words,
$$
\dfrac{v(g(\t))}{\deg(g)}\le \mx\left\{\dfrac{v(h_i(\t))}{\deg(h_i)}\ \Big| \ 1\le i\le t\right\}.
$$
Therefore, (i) and (ii) hold for $g$ if they hold for $h_1,\dots,h_t$.\e

Finally, let us show that if conditions (i), (ii) hold for all monic, irreducible separable polynomials, then both conditions hold for all monic, irreducible polynomials.

Let $g\in\kx$ be monic and irreducible, but inseparable. Let $g\sep=g+\pi x$, for $\pi\in K^*$ with $v(\pi)$ sufficiently large. As mentioned in the proof of Lemma \ref{InsepSep}, $g\sep$ is a separable polynomial of the same degree. Since (i) and (ii) hold for all irreducible factors of $g\sep$, they hold for $g\sep$ too. Hence, if  $v(\pi)$ is sufficiently large, both conditions hold for $g$.
\end{proof}

\subsection{Complete distinguished chains}\label{subsecDChain=Okutsu}

\begin{definition}\label{Defcdchain}
Let $\al_0,\al_1,\dots,\al_r,\t=\al_{r+1}\in\kb$ be algebraic elements such that 
$$
1=\deg_K(\al_0)<\cdots<\deg_K(\al_r)<\deg_K(\t).
$$

We say that $[\al_0,\al_1,\dots,\al_r]$ is a \emph{complete distinguished chain} for $\t$ if $\al_i,\,\al_{i+1}$ is a distinguished pair, for all $0\le i\le r$.
\end{definition}

The next result follows immediately from Theorem \ref{min=min}.

\begin{theorem}\label{CDC=OS}
Let $F\in\kx$ be the minimal polynomial of $\t\in\kb\setminus K$ over $K$.

\begin{enumerate}
\item[(1)] Let $[\phi_0,\dots,\phi_r]$ be an Okutsu frame of $F$. Take $\al_i\in\op{Z}(\phi_i)$ such that $$v(\t-\al_i)=\mx\{v(\t-\al_i')\mid \al'_i\in\op{Z}(\phi_i)\},\qquad 0\le i\le r.$$ Then, $[\al_0,\dots,\al_r]$ is a complete distinguished chain for $\t$. 
\item[(2)] Let $[\al_0,\dots,\al_r]$ be a complete distinguished chain for $\t$. Let $\phi_0,\dots,\phi_r$ be the minimal polynomials of $\al_0,\dots,\al_r$ over $K$, respectively. 

Then, $[\phi_0,\dots,\phi_r]$ is an Okutsu frame of $f$.
\end{enumerate}
\end{theorem}

The next result follows immediately from Theorems \ref{Ok=D1} and \ref{CDC=OS}.

\begin{theorem}[Aghigh-Khanduja \cite{MainInvSudesh,CHF}]\label{defless=OS}
An algebraic element $\t\in\kb$ admits a complete distinguished chain over $K$ if and only if it is defectless over $K$.  
\end{theorem}

\begin{definition}\label{DefOkseq}
Let $\al_0,\al_1,\dots,\al_r,\t=\al_{r+1}\in\kb$ be algebraic elements such that
$$1=\deg_K(\al_0)<\cdots<\deg_K(\al_r)<\deg_K(\t).$$
We say that $[\al_0,\al_1,\dots,\al_r]$ is a \emph{complete Okutsu sequence} for $\t$ if the following conditions hold for all $\be\in\kb$ and all $0\le i\le r$: 
\begin{enumerate}
\item[(1)] \ $\deg_K(\be)<\deg_K(\al_{i+1})\ \imp\ v(\t-\be)\le v(\t-\al_i)$.
\item[(2)] \ $\deg_K(\be)<\deg_K(\al_{i})\ \imp\ v(\t-\be)< v(\t-\al_i)$.
\end{enumerate}
 \end{definition}

Given a distinguished pair  $\al,\,\t$, we have 
\begin{equation}\label{talpha}
\be\in\kb,\ \deg_K(\be)<\deg_K(\al)\ \imp\ v(\t-\be)=v(\al-\be).
\end{equation}
Indeed, by the definition of distinguished pair, $v(\t-\be)<v(\t-\al)$. This remark will be useful to compare Okutsu sequences and complete distinguished chains.

%

\begin{lemma}\label{Ok=CDC2}
A sequence $[\al_0,\al_1,\dots,\al_r]$ of elements in $\kb$ is a complete distinguished chain for $\t=\al_{r+1}$ if and only if it is a complete Okutsu sequence for $\t$.
\end{lemma}

\begin{proof}
Let $\be\in\kb$ with $\deg_K(\be)<\deg_K(\t)$. Suppose that $[\al_0,\al_1,\dots,\al_r]$ is a complete distinguished chain for $\t$. By definition, for all $0\le i\le r$, it holds:
\begin{enumerate}
\item[(i)\,] \ $\deg_K(\be)<\deg_K(\al_{i+1})\ \imp\ v(\al_{i+1}-\be)\le v(\al_{i+1}-\al_i)$.
\item[(ii)] \ $\deg_K(\be)<\deg_K(\al_{i})\ \imp\ v(\al_{i+1}-\be)< v(\al_{i+1}-\al_i)$.
\end{enumerate}

If $i=r$, then $\al_{i+1}=\t$. If $i<r$, we have that $v(\t-\be)=v(\al_{i+1}-\be)$. Thus, in both cases, the conditions of Definition \ref{DefOkseq} coincide with (i) and (ii). Consequently, $[\al_0,\al_1,\dots,\al_r]$ is an Okutsu sequence for $\t$.\e

Conversely, suppose that $[\al_0,\al_1,\dots,\al_r]$ is an Okutsu sequence for $\t$. The conditions of  Definition \ref{DefOkseq} for $i=r$ show that $\al_r,\,\t$ is a distinguished pair. By (\ref{talpha}), 
$$v(\t-\be)=v(\al_r-\be),\qquad v(\t-\al_j)=v(\al_r-\al_j),
$$
for all $0\le j<r$ and all $\be\in\kb$ with $\deg_K(\be)<\deg_K(\al_{j+1})$. Therefore, the sequence $[\al_0,\dots,\al_{r-1}]$ is a complete Okutsu sequence for $\alpha_r$. The previous argument shows that $\alpha_{r-1},\,\alpha_r$ is a distinguished pair. An iterate argument proves that $[\al_0,\al_1,\dots,\al_r]$ is a complete distinguished chain for $\t$. 
\end{proof}

If $[\al_0, \ldots, \al_r]$ is a complete Okutsu sequence  for $\t\in\kb$, then $[\al_0, \ldots, \al_i]$ is a complete Okutsu sequence  for $\al_{i+1}$, for all $1\le i<r$, because this property is obviously true for complete distinguished chains.

\section{Main invariant of quasi-tame algebraic elements}\label{secTame}

In this section, we compute several invariants attached to quasi-tame algebraic elements. To this purpose, Okutsu sequences are a more feasible tool than complete distinguished chains.

\begin{definition}\label{deftame}
Let $\t=\al_{r+1}\in\kb$ be defectless, and let  $[\al_0, \ldots, \al_r]$ be a complete Okutsu sequence for $\t$.
Denote $L=K(\t)$ and let $k_L$ be the residue class field of $(L,v)$. 

We say that $\t\in\kb$ is \emph{tame} if 
\begin{itemize}
\item $k_L/k$ is separable, and
\item the ramification index $e(L/K)$ is not divisible by $\chr(K)$.
\end{itemize}

We say that $\t$ is \emph{quasi-tame} if it is separable and $\al_r$ is tame.
\end{definition}

It is easy to check that a tame $\t$ is necessarily separable over $K$.

Let $K^s\subset \kb$ be the separable closure of $K$ in $\kb$. The subgroup 
$$
G\ram(K)=\left\{\sigma\in\gal(K^s/K)\mid v(\sigma(c)-c)>v(c),\,\forall\, c\in (K^s)^*\right\}
$$
is the \emph{ramification subgroup} of $G$. Its fixed field $K\ram=(K^s)^{G\ram}$ is named the \emph{ramification field} for the extension $K^s/K$. This field is the unique maximal tame extension of $K$ in $\kb$. More precisely, for any algebraic extension $L/K$, the subfield $L\cap K\ram$ is the unique maximal tame extension of $K$ in $L/K$.\e 

From now we consider $[\al_0, \ldots, \al_r]$ a complete Okutsu sequence for $\t=\al_{r+1} \in \kb$ and we shall usually denote
$$
\delta_0=v(\t-\al_0)<\dots<\delta_r=v(\t-\al_r)<\delta_{r+1}=v(\t-\al_{r+1})=\infty.
$$

By Lemma \ref{Ok=CDC2}, we have $\delta_i=v(\al_{i+1}-\al_i)=\delta(\al_{i+1})$, for all $0\le i\le r$.

The next result is inspired in the original ideas of Okutsu \cite{Ok,okutsu}.

\begin{proposition}\label{Hi} Let $[\al_0, \ldots, \al_r]$ be a complete Okutsu sequence for a separable $\t=\al_{r+1} \in K^s$. Consider a separable $\be \in K^s$ such that
$$\deg(\be)=m_{i}, \ \ \ v(\t-\be)> \delta_{i-1},$$
for some $1\le i\le r+1$.
Let $M/K$ be any finite Galois extension containing $K(\t,\be)$. Let $G=\gal(M/K)$ and consider the subgroups
$$
H_i=\{ \sigma \in G \mid v(\t- \sigma(\t))> \delta_{i-1})\} \supset \overline{H}_i=\{ \sigma \in G \mid v(\t- \sigma(\t))\geq \delta_{i}\}.
$$
Let $M^{H_i} \subset M^{\overline{H}_i}\subset M$ be the respective fixed fields. Finally, let $V$ be the maximal tame subextension of $K(\be)/K$. Then,
$$V\subset M^{H_i}\subset K(\t)\cap K(\be).$$
Moreover, if $v(\t-\be)=\delta_{i}$ then $V\subset M^{H_i}\subset M^{\overline{H}_i}\subset K(\t)\cap K(\be)$.
\end{proposition}

\begin{proof}
First, let us show that $M^{H_i}\subset K(\t)\cap K(\be)$. For this, it suffices to show that all  $\sigma \in G$ fixing $\t$ or $\be$ belong to $H_i$. 

If $\sigma(\t)=\t$, then $\sigma \in H_i$ because $v(\t-\sigma(\t))=\infty> \delta_{i-1}$. If $\sigma(\be)=\be$, then $v(\sigma(\t)-\be)=v(\sigma(\t)-\sigma(\be))=v(\t-\be) > \delta_{i-1}.$ Thus,
$$v(\t-\sigma(\t))\geq \Min \{v(\t-\be),\,v(\be-\sigma(\t))\}> \delta_{i-1}.$$

In the case $v(\t-\be)=\delta_{i}$, the same argument shows that $M^{\overline{H}_i}\subset K(\t)\cap K(\be)$.

Finally let us prove that $V\subset M^{H_i}$. Since $V$ is the maximal tame extension of $K(\be)$, we have that $V=K\ram\cap K(\be)$, so we must prove that 
$$H_i \subset \{\sigma \in G \mid v(\sigma(c)-c)>v(c),\quad \forall c \in K(\be)^*\}.$$

Take  $\sigma \in H_i$. Any $c \in K(\be)^*$ can be written as $c=g(\be)$ for some $g \in K[x]$ with  $\deg(g)<m_{i}$. By the minimality of $m_{i}$, for any root $\xi$ of $g$ we have $v(\t-\xi)\leq \delta_{i-1}$. Hence,
$
v(\be-\xi)=\mn\left\{v(\be-\t),\,v(\t-\xi)\right\}=v(\t-\xi)\le \delta_{i-1}
$.

Write $g(x)= a\prod_{\xi\in\op{Z}(g)} (x-\xi)$. Then,
$$\dfrac{g(\sigma(\be))}{g(\be)}= \prod_\xi \dfrac{\sigma(\be)-\xi}{\be-\xi} = \prod_\xi \left(1+ \dfrac{\sigma(\be)-\be}{\be-\xi} \right).$$

Since $\sigma \in H_i$, we have
$$
v\left(\sigma(\be)-\be\right)\geq \Min \left\{v\left(\sigma(\be)-\sigma(\t)\right),\,v\left(\sigma(\t)-\t\right),\, v\left(\t-\be\right)\right\} > \delta_{i-1}.
$$ Since $v(\be-\xi)=v(\t-\xi)\le \delta_{i-1}$, this implies $v((\sigma(\be)-\be)/(\be-\xi)) >0$. Hence, 
$$v \left(\dfrac{\sigma(c)}{c}-1\right)= v \left(\dfrac{g(\sigma(\be))}{g(\be)} -1 \right) >0.$$ 
This proves that $V\subset M^{H_i}$.
\end{proof}

\begin{lemma}\label{diameter+altame}
Let $[\al_0, \ldots, \al_r]$ be a complete Okutsu sequence for $\t \in \kb$. 
\begin{enumerate}
\item[(1)] \ If $F\in\kx$ is the minimal polynomial of $\t$, then $v(\t-\t')\ge v(\t-\al_0)$, for all $\t'\in\op{Z}(F)$.
\item[(2)] \ If $\t$ is quasi-tame over $K$, then   $\al_1, \ldots, \al_r$ are tame over $K$.
\end{enumerate}
\end{lemma}

\begin{proof}
Since $\al_0\in K$, we have $v(\t-\al_0)=v(\t'-\al_0)$ for all $\t'\in Z(F)$, by the henselian property. This proves (1):
$$v(\t-\t')\ge \Min \left\{v(\t-\al_0),\, v(\t'-\al_0)\right\}=v(\t-\al_0).$$

By Lemma \ref{Ok=CDC2}, $[\al_0, \ldots, \al_{i-1}]$ is a complete Okutsu sequence for $\al_i$. Hence, all $\al_i$ are defectless by Theorem \ref{defless=OS}.

As indicated in (\ref{towerki}), we have a tower of finite extensions of $k$:
$$
k=k_{\phi_0}\subset k_{\phi_1}\subset\cdots\subset k_{\phi_r}.
$$
Thus, the assumption that $k_{\phi_r}/k$ is separable implies that all $k_{\phi_i}/k$ are separable too.

Finally, (\ref{ephi}) shows that
$$
1=e(\phi_0)\mid\cdots \mid e(\phi_i)\mid\cdots \mid e(\phi_r).
$$
Thus, if $e(\phi_r)$ is not divisible by the characteristic of $K$, all ramification indices $e(\phi_i)$ have the same property. This proves that $\al_1,\dots,\al_r$ are tame. 
\end{proof}

\begin{theorem}\label{Valuesti}
Let $\t\in\kb$ be quasi-tame of degree $n=\deg_K(\t)>1$.
Let $[\al_0, \ldots, \al_r]$ be a complete Okutsu sequence for $\t=\al_{r+1}$,  and denote
$$
m_i=\deg_K(\alpha_i),\qquad \delta_i=v\left(\t-\alpha_i\right),\qquad 0\le i\le r+1.
$$

Then, it holds:
\begin{enumerate}
\item[(1)] \ $K=K(\al_0) \subset K(\al_1)\subset\cdots \subset K(\al_r) \subset K(\t)$.
\item[(2)] \ The following multisets of cardinality $n-1$ coincide:
$$
\Omega(\t)=\left\{v\left(\t-\t'\right) \mid \t'\in \op{Z}(f), \ \t'\ne\t\right\}=\left\{\delta_0^{t_0},\dots,\delta_r^{t_r}\right\},
$$
where \ $t_i=(n/m_i)-(n/m_{i+1})$ \ for all $0\le i\le r$.
\item[(3)] \ $\delta(\t)=\om(\t)=\delta_r$.
\end{enumerate}
\end{theorem}

\begin{proof} Let $M/K$ be a finite Galois extension of $K$ containing $K(\t,\al_1,\dots,\al_r)$, and denote $G=\gal(M/K)$.

Fix an index $0\le i\le r$. Since $\deg_K(\al_i)=m_i$ and $v(\t-\al_i)=\delta_i$, Proposition \ref{Hi} applied to $\be=\al_i$ shows that $V_i\subset M^{H_{i}}\subset M^{\overline{H}_{i}}\subset K(\al_i)\cap K(\t),
$ where $V_i$ is the maximal tame subextension of $K(\al_i)$.

By Lemma \ref{diameter+altame}, $K(\al_i)/K$ is tame, so that $V_i=K(\al_i)$. Therefore,
\begin{equation}\label{HH} 
V_i=M^{H_{i}}= M^{\overline{H}_{i}}=K(\al_i)\subset K(\t).
\end{equation}

Now, denote $H_0:=G$ and consider the chain of subgroups
$$
G=H_0\supset H_1\supset\cdots\supset H_r\supset H_{r+1}=\gal(M/K(\t)).
$$
The corresponding chain of fixed fields is that given in item (1).\e

Moreover,  (\ref{HH}) implies 
$$
\left(H_i:H_{i+1}\right)=[K(\al_{i+1})\colon K(\al_i)]=m_{i+1}/m_i>1, \qquad 0\le i\le r,
$$
so that all inclusions in the chain of subgroups are strict. Hence, for any 
$\sigma \in G\setminus H_{r+1}$, there exists a unique $0\le i\le r$ such that 
$\sigma\in \overline{H}_{i}=H_i$, and  $\sigma\not\in H_{i+1}.$
If $i>0$, then $v\left(\t-\sigma(\t)\right)=\delta_i$, by the definition of the subgroups $\overline{H}_{i}$ and $H_{i+1}$.

If $i=0$, then $\sigma\not\in H_1$ implies $v\left(\t-\sigma(\t)\right)\le\delta_0$. By Lemma \ref{diameter+altame}, we deduce that $v\left(\t-\sigma(\t)\right)=\delta_0$ in this case too.

Therefore, the underlying set of the multiset $\Omega(\t)$ is the set $\left\{\delta_0,\dots,\delta_r\right\}$.\e

Now, it remains to find a concrete formula for the multiplicity $t_i$ of each value $\delta_i$.

Let $F\in\kx$ be the minimal polynomial of $\t$ over $K$. The natural action of $G$ on $\op{Z}(F)$ induces a bijection:
$$
G/\gal(M/K(\t))\lra \op{Z}(F),\qquad \sigma\longmapsto \sigma(\t).
$$
For any $0\le i\le r$, the restriction of this bijection to the subgroup $H_i/\gal(M/K(\t))$ determines a bijection:
$$
H_i/\gal(M/K(\t))\lra Z_i(F):=\left\{\t'\in \op{Z}(F)\mid v\left(\t-\t'\right)\ge\delta_i\right\}.
$$
Hence, the multiplicity $t_i$ is equal to:
\begin{align*}
t_i=\#Z_i(f)-\#Z_{i+1}(f)&\;=\#H_i/\gal(M/K(\t))-\#H_{i+1}/\gal(M/K(\t))\\&\;=[K(\t)\colon K(\al_i)]-[K(\t)\colon K(\al_{i+1})]=\dfrac{n}{m_i}-\dfrac{n}{m_{i+1}}.
\end{align*}
This ends the proof of items (2) and (3).
\end{proof}

We end this section with an explicit formula for the main invariant $\delta(\t)=\om(\t)$ in terms of the discrete invariants, described in section \ref{secIndVals}, attached to the inductive valuation corresponding to the minimal polynomial of $\t$ over $K$.

\begin{proposition}\label{lambdas}
With the above notation, if $\al_r$ is tame over $K$, then
\begin{equation}\label{ci}
\delta_i=\lambda_0+\cdots+\lambda_i,\qquad 0\le i \le r.
\end{equation}
\end{proposition}

\begin{proof}
Let us prove the formula by a recurrent argument on $i$. 
For $i=0$, we have $\phi_0=(x-\al_0)$ and
$$\la_0=\ga_0= v(\phi_0(\t))=v(\t-\al_0)=\delta_0.$$

Now, suppose that $i>0$ and $\delta_j=\lambda_0+\cdots+\lambda_j$ for all $j<i$. Let us prove that \eqref{ci} holds for $i$. 

We claim that 
\begin{equation}\label{claimLast} 
v(\phi_i(\t))=\delta_i+t_0\delta_0+ \cdots + t_{i-1}\delta_{i-1},\qquad t_j=\dfrac{m_i}{m_j}-\dfrac{m_i}{m_{j+1}},\quad 0\le j<i.
\end{equation}

In fact, since $[\al_0,\dots,\al_{i-1}]$ is a complete Okutsu sequence for $\al_i$, Theorem  \ref{Valuesti} yields an equality of multisets:
\begin{equation}\label{mtset} 
\left\{v(\al_i-\xi)\mid \xi\in\op{Z}(\phi_i),\ \xi\ne\al_i\right\}=\left\{\delta_0^{t_0},\dots,\delta_{i-1}^{t_{i-1}}\right\},
\end{equation}
for the multiplicities $t_0,\dots,t_{i-1}$ indicated in (\ref{claimLast}).

Now, for each $\xi\in\op{Z}(\phi_i)$, $\xi\ne\al_i$, we have
\begin{equation}\label{t=ali} 
v(\t-\xi)=\mn\{v(\t-\al_i),\,v(\al_i-\xi)\}=v(\al_i-\xi),
\end{equation}
because $v(\t-\al_i)=\delta_i$, while 
$v(\al_i-\xi)\le \om(\al_i)=\delta(\al_i)=\delta_{i-1}<\delta_i.$

The equalities (\ref{mtset}) and (\ref{t=ali}) prove the claimed identity (\ref{claimLast}), because
$$
v(\phi_i(\t))=v(\t-\al_i)+\sum_{\xi\in\op{Z}(\phi_i),\;\xi\ne\al_i}v(\t-\xi)=\delta_i+t_0\delta_0+ \cdots + t_{i-1}\delta_{i-1}.
$$

Finally, from (\ref{claimLast}) and \eqref{recurrence} we deduce
$$
 \delta_i+t_0\delta_0+ \cdots + t_{i-1}\delta_{i-1}=\ga_i= \dfrac{m_i}{m_0}\,\la_0+\cdots+\dfrac{m_i}{m_i}\,\la_i,
$$
from which we may express $\delta_i$ as
$$
\delta_i=\dfrac{m_i}{m_0}\,\la_0+\cdots+\dfrac{m_i}{m_i}\,\la_i-t_0\delta_0- \cdots - t_{i-1}\delta_{i-1}.
$$
By applying the induction hypothesis, we may express $\delta_i$ as a linear combination $$\delta_i=a_0\la_0+\cdots+a_{i-1}\la_{i-1}+\la_i,$$ where, for $j<i$, each coefficient $a_j$ takes the value:
$$
a_j=\dfrac{m_i}{m_j}-t_j-t_{j+1}-\cdots -t_{i-1}=\dfrac{m_i}{m_i}=1.
$$
This ends the proof of the proposition.
\end{proof}

The next result follows from Theorem \ref{Valuesti}, Proposition \ref{lambdas} and equation (\ref{recurrence}).

\begin{corollary}\label{last}
If $\t\in\kb$ is quasi-tame, then
$$
\delta(\t)=\om(\t)=\la_0+\cdots+\la_r=\ga_r-\sum_{0\le i<r}\dfrac{m_{i+1}-m_i}{m_i}\,\ga_i.
$$
\end{corollary}

The last equality yields an expression for $\om(\t)$ which is equivalent to a formula of Brown-Merzel in \cite{BrownMerzel}.

\end{document}